\documentclass[10pt,oneside,reqno]{amsart}
\usepackage{setspace,caption}

\usepackage{amsaddr}
\usepackage{amsmath}
\usepackage{amsthm}
\usepackage{amssymb}
\usepackage{mathrsfs}
\usepackage{verbatim}

\usepackage{xcolor}
{%
	\setcounter{enumi}{0}

	\begin{enumerate}}%
	{\end{enumerate} }

{%
	\setcounter{enumi}{0}

	\begin{enumerate}}%
	{\end{enumerate} }

\usepackage[colorlinks=true,linkcolor=blue,citecolor=red]{hyperref}

\usepackage[nameinlink]{cleveref}

\newtheorem{theorem}{Theorem}[section]
\newtheorem{lemma}[theorem]{Lemma}
\newtheorem{proposition}[theorem]{Proposition}
\newtheorem{corollary}[theorem]{Corollary}

\theoremstyle{definition}
\newtheorem{definition}[theorem]{Definition}
\newtheorem{remark}[theorem]{Remark}

\numberwithin{equation}{section}

\allowdisplaybreaks

\newcommand*\R{\mathbb{R}}
\newcommand*\N{\mathbb{N}}
\newcommand*\rn{\mathbb{R}^n}

\newcommand*\omegainfty{\Omega_{\infty}}

\newcommand*\re{\mathbb{R}}
\newcommand*\sph{\mathbb{S}}


\begin{document}

	\title[fractional $p$-Poincar\'e]{On the best constant in fractional $p$-Poincar\'e inequalities on cylindrical domains}
	
	\subjclass[2020]{26D10; 35R09; 46E35; 49J40.}
	
	\author{Kaushik Mohanta$^1$ and Firoj Sk$^2$}
	\address{Indian Institute of Technology Kanpur, India}
	\email{$^1$kmohanta@iitk.ac.in (Corresponding author)}

	\email{$^2$firoj@iitk.ac.in}


	\begin{abstract}
		We investigate the best constants for the regional fractional $p$-Poincar\'e inequality and the fractional $p$-Poincar\'e inequality in cylindrical domains. For the special case $p=2$, the result was already known due to Chowdhury-Csat\'{o}-Roy-Sk [\textit{Study of fractional {P}oincar\'{e} inequalities on unbounded
				domains, Discrete Contin. Dyn. Syst., 41(6), 2021}]. We addressed the asymptotic behaviour of the first eigenvalue of the nonlocal Dirichlet p-Laplacian eigenvalue problem when the domain is becoming unbounded in several directions. 
	\end{abstract}
	
	\maketitle



	\maketitle
	
	\section{introduction}
	
	In the theory of partial differential equations, Poincar\'e inequality has always played an important role. In recent years, the study of various nonlocal analogues of Poincar\'e inequality has seen a steep surge. For the particular case $p=2$, Chowdhury-Csat\'o-Roy-Sk \cite{CCRS} have found the best constants for fractional Poincar\'e inequalities in certain unbounded domains. The aim of this article is to generalise \cite{CCRS} for any $p\in(1,\infty)$.

	For any open set $ \Omega\subseteq\rn$, $0<s<1$, $1\leq p<\infty$, we define the fractional Sobolev space
	$$	
	W^{s,p}(\Omega):=\left\{u\in L^p(\Omega):[u]_{s,p,\Omega}<\infty\right\},
	$$
	where
	\begin{equation}\label{seminorm}
	[u]_{s,p,\Omega}:=\left(\frac{C_{n,s,p}}{2}
	\int_{\Omega}\int_{\Omega}\frac{|u(x)-u(y)|^p}{|x-y|^{n+sp}}dxdy\right)^\frac{1}{p},
	\end{equation}
	\noindent is the so called Gagliardo seminorm. The constant $C_{n,s,p}$ \cite{AnWa} is given by
	\begin{equation}\label{const flap}
	C_{n,s,p}=\frac{ sp\;2^{2s-1}\Gamma\left(\frac{ n+sp}{2}\right)}{2\pi^{\frac{ n-1}{2}}\Gamma(1-s)\Gamma\left(\frac{ p+1}{2}\right)}.
	\end{equation}
	We endow this space with the so-called fractional Sobolev norm, given by 
	$$
	||u||_{s,p,\Omega}:=\left(\|u\|_{L^p(\Omega)}^p+[u]_{s,p,\Omega}^p
	\right)^\frac{1}{p}.
	$$
	At this point, we would like to introduce two more Banach spaces, directly related to the fractional Sobolev spaces $W^{s,p}(\Omega)$ defined above, which will be useful in framing the problem dealt in this article. The spaces $W^{s,p}_{\Omega}(\rn)$
	and
	$W^{s,p}_0(\Omega)$ denote the closures of $C_c^\infty(\Omega)$ with respect to the norms  $\left(\|u\|_{L^p(\Omega)}^p+[u]_{s,p,\rn}^p\right)^\frac{1}{p}$ and $||\cdot||_{s,p,\Omega}$ respectively. The Gagliardo seminorm and the fractional Sobolev norm are also important tools for studying the fractional $p$-Laplacian operator, defined by
	$$
	(-\Delta_{n,p})^su(x):=C_{n,s,p}\lim_{\substack{\epsilon\to 0}}\int_{\rn\setminus B_{\epsilon}(x)}\frac{|u(x)-u(y)|^{p-2}(u(x)-u(y))}{|x-y|^{n+sp}}dy,\;\;\;x\in\rn.
	$$
	Note that while defining the Gagliardo seminorm and the fractional $p$-Laplacian operator, in the existing literature, the constant $C_{n,s,p}$ is often ignored. However, we shall take this constant into account as this will be convenient while studying the best constants in fractional Poinca\'re inequalities (see \cref{rev-rem} bellow). We refer the reader to \cite{ada,BrSa,hhg,FiSeVa,FrPa,Lind, DeBoRi} for basic results regarding the fractional Sobolev spaces and the fractional $p$-Laplacian operator. 
	
	\noindent The \textit{regional fractional Poincar\'e constant} $P^1_{n,s,p}(\Omega)$, and \textit{fractional Poincar\'e constant} $P^2_{n,s,p}(\Omega)$  are defined as follows:
	$$
	P^1_{n,s,p}(\Omega):=\inf_{\substack{u\in W_0^{s,p}(\Omega) 
			\\ u\neq 0}}
	\frac{[u]_{s,p,\Omega}^p}{\displaystyle\int_{\Omega}|u|^p}, 
	\;\text{ and }\; 
	P^2_{n,s,p}(\Omega):=\inf_{\substack{u\in W^{s,p}_{\Omega}(\rn) 
			\\ u\neq 0}}
	\frac{[u]_{s,p,\R^n}^p}{\displaystyle\int_{\Omega}|u|^p}.
	$$
	\begin{definition}
		Let $\Omega$ be an open set in $\rn$. We say that
		\begin{itemize}
			\item the \textit{regional fractional Poincar\'e inequality} \textbf{(RFPI)} holds if $P^1_{n,s,p}(\Omega)>0$.
			\smallskip
			\item the \textit{fractional Poincar\'e inequality} \textbf{(FPI)}  holds if $P^2_{n,s,p}(\Omega)>0$. 
		\end{itemize}
	\end{definition}
	For $\ell>0$, set $\Omega_{\ell}:=\ell\omega_1\times\omega$, where $\omega_1$ and $\omega$ are bounded open subsets of $\R^m$ and $\R^{n-m}$ respectively, and consider the nonlocal Dirichlet $p$-Laplacian eigenvalue problem on $\Omega_\ell$:
	
	\begin{equation}\label{nonlocal EV}
	\begin{cases}
	(-\Delta_{n,p})^su_\ell=P_{n,s,p}^2(\Omega_{\ell})|u_\ell|^{p-2}u_\ell \text{ in }\Omega_{\ell},
	\\
	u_\ell=0 \text{ in }\R^{n}\setminus\Omega_{\ell},
	\end{cases}
	\end{equation} 
	and the corresponding cross section eigenvalue problem of \cref{nonlocal EV}
	\begin{equation}\label{cross section EV}
	\begin{cases}
	(-\Delta_{n-m,p})^su=P_{n-m,s,p}^2(\omega)|u|^{p-2} u \text{ in }\omega,
	\\
	u=0 \text{ in }\R^{n-m}\setminus \omega.
	\end{cases}
	\end{equation}

	In the existing literature, the \textbf{RFPI} and \textbf{FPI}, in unbounded domains, are not much explored yet. However, it is well-known that the \textbf{FPI} holds true that is $P^2_{n,s,p}(\Omega)>0$, when $\Omega$ is a bounded domain \cite{brasco2}. In \cite{BrCi}, the authors have shown that the \textbf{FPI} holds true for any domain which is bounded in one direction, although the question of best fractional Poincar\'e constant remained unattended (see also \cite{Yer}). In the special case $p=2$, it is known (see \cite{CCRS}) that the best  fractional Poincar\'e constants $P_{n,s,2}^1(\mathbb{R}^{n-1}\times (-1,1))$ and $P^2_{n,s,2}(\mathbb{R}^{m}\times\omega)$ are equal to that of the cross sections, that is to $P^1_{1,s,2}((-1,1))$ and $P^2_{n-m,s,2}(\omega)$ respectively, where $\omega$ is a bounded domain in $\R^{n-m}$. The first work in this direction, to the best of our knowledge, was done in \cite{CCRS,AmFrMu}. Later, some of the results of \cite{CCRS} were generalized in the Orlicz fractional Sobolev setup in \cite{BMRS}. 
	Regarding the \textbf{RFPI}, it is known that when $\Omega$ is a bounded domain with Lipschitz boundary, \textbf{RFPI} does not hold, that is $P^1_{n,s,p}(\Omega)=0$ if $0<s\leq\frac{1}{p}$ \cite{AnWa, War}. The \textbf{RFPI}, however, holds true, that is $P^1_{n,s,p}(\Omega)>0$ if $\frac{1}{p}<s<1$, when $\Omega$ is any bounded domain in $\rn$. We refer the reader to \cite{AnWa,War} for other related results regarding regional fractional $p$-Laplacian operator.  Regarding the asymptotic behaviour of the first eigenvalue $P^2_{n,s,p}(\Omega_{\ell})$ of \cref{nonlocal EV}, when $\ell\to\infty$, in the linear case, that is for $p=2$, Chowdhury-Roy \cite{ChRo} proved that the first eigenvalue $P^2_{n,s,2}(\Omega_{\ell})$ of \cref{nonlocal EV} converges to the first eigenvalue $P^2_{n-m,s,2}(\omega)$ of \cref{cross section EV}, when $\ell\to\infty$. Different kind of problems were studied regarding the asymptotic behaviour of $P^2_{n-m,s,2}(\Omega_\ell)$ as $\ell\to\infty$; we refer the readers to \cite{ChRoSh,EsRoSk,ChRo,Yer} and the references therein for more relevant information in this direction.

	Our first result depicts, for a cylindrical domain, that the best constant for \textbf{FPI} of the domain and that of the cross-section are the same. Indeed, we have the following result:
	
	\begin{theorem}\label{best constants}
		Let $0<s<1$, $1<p<\infty$ and $\omegainfty=\re^m\times\omega$ in $\rn$ with $1\leq m<n$, where $\omega$ is a bounded open subset of $\re^{n-m}$. Then we have 
		\begin{equation}\label{Eq2}
		P^2_{n,s,p}(\omegainfty)=P^2_{n-m,s,p}(\omega).    
		\end{equation}
		Furthermore, the best fractional Poincar\'e constant $P^2_{n,s,p}(\omegainfty)$, is never achvied.
	\end{theorem}

	\begin{remark}\label{rev-rem}
			Recall that we took into account, the constant $\frac{C_{n,s,p}}{2}$ (defined in \cref{const flap}) while defining the seminorm in \cref{seminorm}. If one ignores this constant in the definition of the seminorm, then one would get an additional multiplicative constant in \cref{Eq2} in place of plain equality in \cref{best constants}.
	\end{remark}
	\smallskip
	
	The strategy for the proof of \cref{best constants}, done in \cref{sec:main}, is somewhat analogous to the local case. It is as follows: the constant $P^2_{n,s,p}(\omegainfty)$ is bounded above by the constant $P^2_{n,s,p}(\omega)$ (see (2) of \cref{useful prop}). For the special case $p=2$, the regularity of the first eigenfunction of \cref{cross section EV} is extensively used for the bounded below case. However, for general $p>1$, we do not have such regularity theory of the first eigenfunction of \cref{cross section EV}. We use simple approximation argument and discrete Picone inequality (see, \cref{picone}) to prove this part.
	The last part of the theorem is proven via contradiction, where the geometry of the domain has played an important role. Note that the use of discrete Picone inequality forces us to exclude the case $p=1$ from the statement of the \cref{best constants}.\smallskip

	Next, we deal with the case of \textbf{RFPI}. As above, we show that the best constant for \textbf{RFPI} for a strip is equal to that of its cross-section.
	
	\begin{theorem}\label{Poincare 1} Let $0<s<1$, $1\leq p<\infty$ and
		$\omegainfty=(-1,1)\times\R^{n-1}\subset\rn$, then we have the following:
		\begin{enumerate}
			\item $P^1_{n,s,p}(\omegainfty)=P^1_{1,s,p}((-1,1))=0$, if $0<s\leq\frac{1}{p}$.
			\item
			$
			P^1_{n,s,p}(\omegainfty)=P^1_{1,s,p}((-1,1)).
			$ Consequently, $P^1_{n,s,p}(\omegainfty)>0$, if $\frac{1}{p}<s<1$. 
		\end{enumerate}
	\end{theorem}
	\smallskip

	The proof of  \cref{Poincare 1} goes along the same line as it is done in \cite{CCRS} for $p=2$ but with necessary modifications. However, the method of the proof differs significantly from that of \cref{best constants} and hence, in this case, we must stick to the case $m=1$, $\omega=(-1,1)$,.
	\smallskip

	Finally, we come to our last main result, which shows the asymptotic behaviour of the first eigenvalue of \cref{nonlocal EV}. 
	
	\begin{theorem}\label{l to infinity} 
		Let $0<s<1$, $1<p<\infty$, $\ell>0$ and $\Omega_\ell=\ell\omega_1\times\omega$ in $\rn$ with $1\leq m<n$, where $\omega_1,\ \omega$ are bounded open subsets of $\re^m$ and $\re^{n-m}$ respectively. We then have
		$$P_{n-m,s,p}^2(\omega)\leq P^2_{n,s,p}(\Omega_{\ell})\leq P_{n-m,s,p}^2(\omega)+\frac{C_1}{\ell^s}+\frac{C_2}{\ell^{sp}},$$
		where $C_1,C_2>0$ are constants independent of $\ell.$ Furthermore, if $\Omega_{\infty}=\bigcup_{\ell>0}\Omega_{\ell}$ 
		$$\lim\limits_{\ell\to\infty}P^2_{n,s,p}(\Omega_{\ell})=P^2_{n-m,s,p}(\omega)=P^2_{n,s,p}(\Omega_\infty).$$
	\end{theorem}

	\smallskip

	This article is organized in the following way: In \cref{sec: known results} we recall some results, already known in the literature. In \cref{sec:main} we give proofs of \cref{Poincare 1,best constants,l to infinity}.

	\section{Some known results and conventions} \label{sec: known results}
	Here we briefly discuss the notations that we shall use throughout the paper.
	
	\begin{itemize}
		\item $s$ will always be understood to be in $(0,1)$.
		\item For any positive integer $n$ and a measurable set $\Omega\subset\R^n$ we write $\mathcal{L}^n(\Omega)$ to denote the Lebesgue measure of $\Omega$, or shortly $|\Omega|$ if $n$ is understood from the context.
		
		\item $B_R(x)$ denotes a ball of radius $R$ centered at $x.$ We shall also write $B_R$ for $B_R(0)$.
		
		\item $\sph^{m-1}$ is the unit sphere in the Euclidean space $\re^m.$
		
		\item $\mathcal{H}^{k}$ denotes the $k$-dimensional Hausdorff measure, so that
		\begin{equation}
		\label{eq:surface of S n-1}
		\mathcal{H}^{n-1}\left(\mathbb{S}^{n-1}\right)=\frac{2\pi^{\frac{n}{2}}}
		{\Gamma\left(\frac{n}{2}\right)},\quad\text{where $\Gamma$ is the standard gamma function}.
		\end{equation}
		
		\item The beta function, for $x,y>0$, is defined by
		\begin{equation}\label{eq:properties of Beta}
		B(x,y):=\int_0^1t^{x-1}(1-t)^{y-1}dt=
		B(x,y)=2\int_0^{\frac{\pi}{2}}(\sin\theta)^{2x-1}(\cos\theta)^{2y-1}d\theta=\frac{\Gamma(x)\Gamma(y)}{\Gamma(x+y)}.
		\end{equation}
		
	\end{itemize}

	We now state some definitions and results, already known in literature, which we shall be using in the subsequent sections in this article. But before defining these, we would like to recall a result which follows from \cite[Lemma~2.7]{DeBoRi}.
	\begin{lemma}\label{bonder}
		Let $\Omega\subset \R^{n}$ be an open set, $p>1$. Then
		\begin{multline*}
		\int_{\R^{n}}\int_{\R^{n}}\frac{ |u(x)-u(y)|^{p-2}(u(x)-u(y))\psi(x)}{|x-y|^{n+sp}}dxdy\\ =\frac{1}{2}\int_{\R^{n}}\int_{\R^{n}}\frac{ |u(x)-u(y)|^{p-2}(u(x)-u(y))(\psi(x)-\psi(y))}{|x-y|^{n+sp}}dxdy,
		\end{multline*}
		whenever the integral in the left hand side is finite or $[u]_{s,p,\rn},\ \|\psi\|_{L^p(\Omega)}<\infty$. Here the integral in the LHS is to be understood in the principle value (P.V.) sense.
	\end{lemma}
	\begin{proof}
		Let $u,\psi\in W^{s,p}_\Omega (\R^n)$. Set
		\begin{multline*}
		I:=\mbox{P.V.}\int_{\R^{n}}\int_{\R^{n}}\frac{ |u(x)-u(y)|^{p-2}(u(x)-u(y))\psi(x)}{|x-y|^{n+sp}}dxdy\\
		:=\lim\limits_{\epsilon\to 0}\iint_{ |x-y|\geq\epsilon}\frac{ |u(x)-u(y)|^{p-2}(u(x)-u(y))\psi(x)}{|x-y|^{n+sp}}dxdy=\lim\limits_{\epsilon\to 0} I_\epsilon.
		\end{multline*}
		If, $I_\epsilon$ is finite, we can write, $I=\frac{1}{2}\lim\limits_{\epsilon\to 0}(I_\epsilon+I_\epsilon)$, then we make a change of variable, by interchanging $x$ and $y$, in the second $I_\epsilon$ in the RHS and then we add the two terms in the RHS to conclude the equality.
		
		It remains to show that $[u]_{s,p,\rn},\ \|\psi\|_{L^p(\Omega)}<\infty$ implies finiteness of $I_\epsilon$. The following calculation is done in \cite[lemma~2.3]{BrCi}. However, we include it here for the sake of completeness. For a fixed $\epsilon>0$, we have
		\begin{multline*}
		\left|	\iint_{ |x-y|\geq\epsilon}\frac{ |u(x)-u(y)|^{p-2}(u(x)-u(y))\psi(x)}{|x-y|^{n+sp}}dxdy \right|\\
		\leq \iint_{ |x-y|\geq\epsilon}\frac{ |u(x)-u(y)|^{p-1}|\psi(x)|}{|x-y|^{n+sp}}dxdy\\
		= \iint_{ |x-y|\geq\epsilon}\frac{ |u(x)-u(y)|^\frac{p}{p'}|\psi(x)|}{|x-y|^\frac{n+sp}{p'}|x-y|^\frac{n+sp}{p}}dxdy\\
		\leq \left(\iint_{ |x-y|\geq\epsilon}\frac{ |u(x)-u(y)|^p}{|x-y|^{n+sp}}dxdy\right)^\frac{1}{p'}\left( \iint_{ |x-y|\geq\epsilon}\frac{|\psi(x)|^{p}}{|x-y|^{n+sp}}dxdy \right)^\frac{1}{p}\\
		\leq [u]_{s,p,\rn}^\frac{p}{p'} \left( \iint_{ |x-y|\geq\epsilon}\frac{|\psi(x)|^{p}}{|x-y|^{n+sp}}dxdy \right)^\frac{1}{p}
		=C [u]_{s,p,\rn}^\frac{p}{p'}\|\psi\|_{L^p(\Omega)}\epsilon^{-s}<\infty.
		\end{multline*}
	
	\end{proof}
	\smallskip
	
	Using this lemma, we define
	
	\begin{definition}
		\label{weak formulation}
		Let $\omega\subset \R^{n-m}$ be a bounded open set. A function $u\in W^{s,p}_{\omega}(\R^{n-m})$ is said to be a \textit{weak solution} of \cref{cross section EV} if $u$ satisfies
		\begin{multline*}
		\frac{C_{n-m,s,p}}{2}\int_{\R^{n-m}}\int_{\R^{n-m}}\frac{ |u(x)-u(y)|^{p-2}(u(x)-u(y))(\psi(x)-\psi(y))}{|x-y|^{n-m+sp}}dydx\\
		=C_{n-m,s,p}\int_{\R^{n-m}}\int_{\R^{n-m}}\frac{ |u(x)-u(y)|^{p-2}(u(x)-u(y))\psi(x)}{|x-y|^{n-m+sp}}dydx\\
		=P^2_{n-m,s,p}(\omega)\int_{\omega}|u(x)|^{p-2}u(x)\psi(x)dx,\text{ for all }\psi\in W^{s,p}_{\omega}(\R^{n-m}).
		\end{multline*}
		Any such $u$, not identically zero, is also called an \textit{eigenfunction} of \cref{cross section EV}, corresponding to the eigenvalue $P^2_{n-m,s,p}(\omega)$.
	\end{definition}
	
	\begin{lemma}[\protect{see \cite{BrPa}}]\label{simplicity}
		The constant $P^2_{n,s,p}(\omega)$ is the first eigenvalue of the problem \cref{cross section EV}, and the corresponding eigenfunction is  strictly positive in the domain.	Moreover, the corresponding eigenspace is of dimension one. 
	\end{lemma}
	\begin{lemma}[\protect{see \cite[Lemma~2.4]{loss}} ]
		\label{lemma:Loss Sloan}
		Let $p>0,$ $0<s<1$ and $\Omega\subset\R^n$ be a measurable set. Then for any $u\in C_c^{\infty}(\Omega)$
		\begin{align*}
		&2\int_{\Omega}\int_{\Omega}
		\frac{|u(x)-u(y)|^p}{|x-y|^{n+sp}}dxdy
		\smallskip
		\\
		=&\int_{\sph^{n-1}}d\mathcal{H}^{n-1}(w)
		\int_{\{x:\;x\cdot w=0\}}d\mathcal{H}^{n-1}(x)
		\int_{\{\ell:\,x+\ell w\in\Omega\}}
		\int_{\{t:\,x+tw\in\Omega\}}\frac{|u(x+\ell w)-u(x+tw)|^p}{|\ell-t|^{1+sp}}dtd\ell.
		\end{align*}
	\end{lemma}
	
	\begin{lemma}[Discrete Picone inequality, \cite{BrFr}]\label{picone}
		Let $p\in(1,\infty)$ and let $f,g:\rn\to\re$ be two measurable functions with $f\geq 0,\;g>0$, then $L(f,g)\geq 0\text{ in }\rn\times\rn,$
		where
		$$
		L(f,g)(x,y)=|f(x)-f(y)|^p-|g(x)-g(y)|^{p-2}(g(x)-g(y))\left(\frac{ f(x)^p}{g(x)^{p-1}}-\frac{ f(y)^p}{g(y)^{p-1}}\right).
		$$
		The equality holds if and only if $f=\alpha g$ a.e. in $\rn$ for some constant $\alpha.$
	\end{lemma}
	The following result is well known in literature. It follows directly from the definition of the Poincar\'e constant and of the Gagliardo seminorm.
	
	\begin{proposition}\label{prop:elementary properties}
		Let $0<s<1$ and $p\in[1,\infty)$ we have
		\begin{enumerate}
			\item\textbf{(Domain monotonicity:)} If $\Omega_1\subseteq\Omega_2\subset\rn,$ then $P^2_{n,s,p}(\Omega_2)\leq P^2_{n,s,p}(\Omega_1).$
			\smallskip
			
			\item\textbf{(Dilation:)} Let $\Omega\subset\rn$ be an open set and let $u\in W^{s,p}(t\Omega)$, for $t>0.$ We define $v_t(x)=u(tx)\in W^{s,p}(\Omega)$. Then $[u]_{s,p,t\Omega}^p=t^{n-sp}[v_t]_{s,p,\Omega}^p$, and furthermore
			$$P^1_{n,s,p}(\Omega)=t^{n-sp}P^1_{n,s,p}(t\Omega),\text{ and }P^2_{n,s,p}(\Omega)=t^{n-sp}P^2_{n,s,p}(\Omega).$$
		\end{enumerate}
	\end{proposition}
	\smallskip
	
	\begin{remark}
		To the best of our knowledge, the domain monotonicity property for $P^1_{n,s,p}$ is not known in literature.
	\end{remark}
	
	The proof of the following well-known result, in the case $sp<1$ can be found in \cite[Theorem~3.4.3]{Tri} for bounded $C^\infty$-domains. For bounded Lipschitz domains, in the case $sp\leq 1$, this can be found in \cite[Section~2]{dyd}, \cite[Theorem 2.1]{AnWa}, \cite[Example 4.11]{War}. For various related results, we refer the reader to \cite{DyKi}.
	\begin{lemma}\label{two spaces equal}
		Let $p\in[1,\infty)$, and $\Omega$ be an open bounded set in $\rn$ with Lipschitz boundary. Then $W^{s,p}_0(\Omega)=W^{s,p}(\Omega)$ if $0<s\leq\frac{1}{p}$. In particular, in this case, we have $P^1_{n,s,p}(\Omega)=0$.
	\end{lemma}
	\smallskip
	
	\begin{lemma}\label{C1-subset}
		Let $p\in[1,\infty)$, and $\Omega\subset \rn$ be an open set. Then $C_c^1(\Omega)\subset W^{1,p}_0(\Omega) \subset W^{s,p}_{\Omega}(\rn)$.
	\end{lemma}
	\begin{proof}
		To show the first inclusion, let us take an arbitrary $v\in C_c^1(\Omega)$. Then we apply \cite[Lemma~8]{BaMoRo} to get a bounded open set $\Omega_1$ with smooth boundary such that $\mbox{supp}(v)\subset\Omega_1\subset\Omega$. Clearly $v\in W^{1,p}(\Omega_1)$. Then we can say, from the well known trace theorem for Sobolev spaces, that $v\in W^{1,p}_0(\Omega_1) \subset W^{1,p}_0(\Omega)$.
		\smallskip
		
		The last inclusion follows from \cite[Proposition~2.2]{hhg}, which implies that for any $u\in C_c^\infty(\Omega)$, 
		\begin{equation*}
		\int_{\rn}\int_{\rn}\frac{|u(x)-u(y)|^p}{|x-y|^{n+sp}}dxdy \leq C(n,s,p)\left(\int_{ \Omega}|u(x)|^pdx+ \int_{\Omega} |\nabla u(x)|^pdx\right)
		\end{equation*}
		as for any $v\in W^{1,p}_0(\Omega)$, there exists a sequence of functions $v_n\in C_c^\infty (\Omega)$, converging to $v$ in $W^{1,p}_0(\Omega)$. The above inequality then suggests that the same sequence will converge to $v$ in $W^{s,p}_\Omega(\rn)$ as well.
	\end{proof}
	\smallskip
	
	\subsection*{Hyper-spherical Coordinates}\label{hyper} Before moving further, let us recall the hyper-spherical coordinates and derive an equality, which will be used in the forthcoming section. 
	
	\noindent Let
	$$
	A_{n-1}=(0,\pi)^{n-2}\times (0,2\pi)\subset\re^{n-1}.
	$$
	The hyper spherical coordinates  $H=(H_1,\cdots,H_n):A_{n-1}\to \mathbb{S}^{n-1}$ are defined as follows:\\
	for $k=1,\cdots,n$ and $\sigma=(\sigma_1,\cdots,\sigma_{n-1})$
	$$
	H_k(\sigma)=\cos\sigma_k\prod_{l=0}^{k-1}\sin\sigma_l\quad\text{ with the convention }\sigma_0=\frac{\pi}{2}\mbox{ and }\sigma_n=0.
	$$
	An elementary calculation shows
	$
	d_i(\sigma):=\left\langle \frac{\partial H}{\partial \sigma_i},\frac{\partial H}{\partial \sigma_i}\right\rangle =\prod_{l=0}^{i-1}\sin^2\sigma_l>0.
	$
	One can easily verify that the metric tensor, in these coordinates, is diagonal, that is
	$
	g_{ij}(\sigma)=\left\langle\frac{\partial H}{\partial \sigma_i},\frac{\partial H}{\partial \sigma_j}\right\rangle=\delta_{ij}d_i(\sigma),
	$
	(Here $\delta_{ij}$ denotes the usual `Kronecker delta')
	and hence the surface element $g_{n-1}$ is given by
	$$
	g_{n-1}(\sigma)=\sqrt{\det g_{ij}(\sigma)}=\sqrt{\prod_{k=1}^{n-1}d_k(\sigma)}=\prod_{k=1}^{n-1}\prod_{l=0}^{k-1}\sin\sigma_l= \prod_{k=1}^{n-2}(\sin\sigma_k)^{n-k-1}.
	$$

	\section{Proof of Main Results}\label{sec:main}
	
	\begin{lemma}
		\label{reduction formula}
		Let $0<s<1$, $1\leq p<\infty$, and for any $m, n\in\N$ with $1\leq m<n$. Let $C_{n,s,p}$ be the constant as in \cref{const flap}. Then we have the following:
		\smallskip
		
		(i) $
		C_{n,s,p}\Theta_{m,n,p}=C_{n-m,s,p},
		\quad\text{ where }\quad \Theta_{m,n,p}=\mathcal{H}^{m-1}
		\left(\sph^{m-1}\right)\int_0^\infty\frac{t^{m-1}}   
		{(1+t^2)^\frac{n+sp}{2}}\;dt
		$
		\smallskip
		
		(ii) If $a>0$ and $z\in\R^m$ then
		\begin{align*}
		\int_{\R^m}\frac{dx}{\left(1+\frac{|x-z|^2}{a^2}\right)^{\frac{n+sp}{2}}}
		=a^m \Theta_{m,n,p}.
		\end{align*}
		
	\end{lemma}

	\begin{proof}
		(i) Applying the change of variable $t=\tan\theta$ in the expression of $\Theta_{m,n,p}$, followed by \cref{eq:surface of S n-1,eq:properties of Beta}, we obtain 
		\begin{multline*}
		\Theta_{m,n,p} =   
		\mathcal{H}^{m-1}\left(\mathbb{S}^{m-1}\right)
		\int_0^\frac{\pi}{2}(\sin\theta)^{m-1}(\cos\theta)^{n-m+sp-1}\;d\theta\\
		=\frac{1}{2}B\bigg(\frac{m}{2},\frac{n-m+sp}{2}\bigg) 
		\frac{2\pi^{\frac{m}{2}}}{\Gamma(\frac{m}{2})}
		=\frac{\pi^{\frac{m}{2}}\Gamma\left(\frac{n-m+sp}{2}\right)}{\Gamma\left( \frac{n+sp}{2} \right)}.
		\end{multline*}
		From \cref{const flap} we get the desired result.
		\smallskip 
		
		(ii) Taking the change of variable $y=\frac{x-z}{a}$, the identity follows immediately.
	\end{proof}
	
	\begin{proposition}\label{useful prop}
		Let $m,n\in\N$ with $m <n $, $0<s<1$, $1\leq p<\infty$ and $\Omega_{\infty}=\R^m\times\omega$, where $\omega\subset\R^{n-m}$ is a bounded open set. Then we have 
		\begin{enumerate}
			\item $P^1_{n,s,p}(\omegainfty)\leq P^1_{n-m,s,p}(\omega)$.
			\smallskip
			
			\item $	P^2_{n,s,p}(\omegainfty)\leq P^2_{n-m,s,p}(\omega).$
		\end{enumerate}
		
	\end{proposition}
	
	\begin{proof}
		First, we prove (1). Note that if we can show, for any $W\in  C_c^{\infty}(\omega)$ and $\epsilon>0$, that there exists $u\in C_c^{\infty}(\Omega_{\infty})$ such that
		\begin{equation*}
		\frac{[u]_{s,p,\Omega_{\infty}}^p}{\|u\|^p_{L^p(\Omega_{\infty})}}\leq
		\frac{[W]^p_{s,p,\omega}}{\|W\|^p_{L^p(\omega)}}+\epsilon,
		\end{equation*}
		then we are done.
		
		Therefore, we start by choosing, arbitrarily, $W\in  C_c^{\infty}(\omega)$ and $v\in C_c^{\infty}(\R^m)$ which satisfies $\int_{\R^m}|v|^p=1$. We define, for $\ell>0$,
		$v_\ell(x)
		=\ell^{-\frac{m}{p}}v\left(\frac{x}{\ell}\right).
		$
		Clearly $v_\ell\in C_c^{\infty}(\R^m)$ and 
		\begin{equation}
		\label{eq:vl square integral is one}
		\int_{\R^m}|v_{\ell}|^p=1\quad\text{ for all }\ell>0.
		\end{equation}
		We denote the point $x\in\rn$ by $x=(X_1,X_2)$, where $X_1\in \R^m$ and $X_2\in \R^{n-m}$. Now we define
		$$
		u_{\ell}(X_1,X_2)=v_{\ell}(X_1)W(X_2).
		$$
		Note that we can always assume $\|W\|_{L^p(\omega)}=1$ by normalizing $W$ appropriately. Using \cref{eq:vl square integral is one} we get
		$$
		\|u_{\ell}\|_{L^p(\Omega_{\infty})}^p=
		\int_{\R^m}\int_{\omega}|v_{\ell}(X_1)|^p|W(X_2)|^pdX_2\,dX_1
		=\|W\|_{L^p(\omega)}^p=1\;\;\text{ for all }\ell>0.
		$$
		Therefore it only remains to show that, for sufficiently large $\ell$,
		\begin{equation*}
		\label{eq:norm infty leq omeganorm plus eps}
		[u_{\ell}]^p_{s,p,\Omega_{\infty}}\leq
		[W]^p_{s,p,\omega}+\epsilon(\ell),
		\end{equation*}
		where $\lim\limits_{\ell\to \infty}\epsilon(\ell)=0$; again, this will follow immediately, if we can show, after redefining $\epsilon(\ell)$ appropriately,
		\begin{equation*}
		[u_{\ell}]_{s,p,\Omega_{\infty}}\leq
		[W]_{s,p,\omega}+\epsilon(\ell).
		\end{equation*}
		
		\noindent Using the triangle inequality of $L^p(\Omega_\infty\times\Omega_\infty)$-norm, we obtain
		\begin{align}\label{main est}
		[u_\ell]_{s,p,\Omega_{\infty}}
		&=\left( \frac{C_{n,s,p}}{2}\int_{\Omega_\infty}\int_{\Omega_\infty}\frac{|u_\ell(x)-u_\ell(y)|^p}{|x-y|^{n+sp}}dxdy\right)^\frac{1}{p}\nonumber
		\\
		&=\left(\frac{C_{n,s,p}}{2}\int_{\Omega_\infty}\int_{\Omega_\infty}\frac{|v_\ell(X_1)W(X_2)-v_\ell(Y_1)W(Y_2)|^p}{|x-y|^{n+sp}}dxdy\right)^\frac{1}{p}\nonumber
		\\
		&=\left(\frac{C_{n,s,p}}{2}\int_{\Omega_\infty}\int_{\Omega_\infty}\left| \frac{v_\ell(Y_1)\left(W(X_2)-W(Y_2)\right)}{|x-y|^{\frac{n}{p}+s}}+\frac{W(X_2)\left(v_\ell(X_1)-v_\ell(Y_1)\right)}{|x-y|^{\frac{n}{p}+s}}\right|^p dxdy\right)^\frac{1}{p}\nonumber
		\\
		&\leq I_1+I_2,
		\end{align}
		where 
		$$
		I_1=\left(\frac{C_{n,s,p}}{2}\int_{\Omega_\infty}\int_{\Omega_\infty}\frac{ |v_\ell(Y_1)|^p|W(X_2)-W(Y_2)|^p}{|x-y|^{n+sp}}dxdy\right)^{1/p}
		$$
		and
		$$
		I_2=\left(\frac{C_{n,s,p}}{2}\int_{\Omega_\infty}\int_{\Omega_\infty}\frac{ |W(X_2)|^p|v_\ell(X_1)-v_\ell(Y_1)|^p}{|x-y|^{n+sp}}dxdy\right)^{1/p}.
		$$
		\smallskip
		
		We shall now estimate the integrals $I_1$ and $I_2$.
		\smallskip
		
		\noindent\textit{Estimate for $I_1$:} For $X_2\neq Y_2$ and by (ii) of \cref{reduction formula}, we get
		\begin{align*}
		\int_{\R^m}\frac{dX_1}{\left(1+\frac{|X_1-Y_1|^2}{|X_2-Y_2|^2}\right)^{\frac{n+sp}{2}}}
		=|X_2-Y_2|^m \Theta_{m,n,p}\hspace{3mm} \textrm{for any} \  Y_1\in \R^m.
		\end{align*}
		Applying this identity to the definition of $I_1$, together with (i) of \cref{reduction formula} and \cref{eq:vl square integral is one}, we get
		\begin{align*}
		I_1^p=&
		\frac{C_{n,s,p}}{2}
		\int_{\omegainfty}\int_{\omegainfty}\frac{|{v_\ell}
			(Y_1)(W(X_2)-W(Y_2))|^p}{|X_2-Y_2|^{n+sp}
			\left(1+\frac{|X_1-Y_1|^2}{|X_2-Y_2|^2}\right)^{\frac{n+sp}{2}}}dxdy 
		\smallskip
		\\
		=&\frac{C_{n,s,p}}{2} 
		\int_{\omega}\int_{\omega}\frac{|W(X_2)-W(Y_2)|^p}{|X_2-Y_2|^{n+sp}}\int_{\re^m}   
		\left(\int_{\re^m}\frac{dX_1}{\big(1+\frac{|X_1-Y_1|^2}{| 
				X_2-Y_2|^2}\big)^{\frac{n+sp}{2}}}\right)|v_\ell(Y_1)|^pdY_1dX_2dY_2 
		\smallskip
		\\
		=&
		\frac{C_{n,s,p}}{2}\Theta_{m,n,p}\int_{\omega}\int_{\omega}
		\frac{\left|W(X_2)-W(Y_2)\right|^p}{\left|X_2-Y_2\right|^{n-m+sp}}dX_2dY_2
		\int_{\R^m}|v_{\ell}(Y_1)|^pdY_1
		\smallskip
		\\
		=&
		\frac{C_{n-m,s,p}}{2}\int_{\omega}\int_{\omega}
		\frac{\left|W(X_2)-W(Y_2)\right|^p}{\left|X_2-Y_2\right|^{n-m+sp}}dX_2dY_2
		=[W]_{s,p,\omega}^p\,.
		\end{align*}
		\smallskip
		
		\noindent\textit{Estimate  for $I_2$:} We can write $I_2$ as
		\begin{align*}
		I_2^p&=\frac{C_{n,s,p}}{2}\int_{\omegainfty}\int_{\omegainfty}\frac{|({v_\ell}(X_1)-{v_\ell}(Y_1))W(X_2)|^p}{|X_1-Y_1|^{n+sp}\left(1+\frac{|X_2-Y_2|^2}{|X_1-Y_1|^2}\right)^{\frac{n+sp}{2}}}dxdy 
		\smallskip 
		\\
		&=\frac{C_{n,s,p}}{2}
		\int_{\re^m}\int_{\re^m}\frac{|{v_\ell}(X_1)-{v_\ell}(Y_1)|^p}
		{|X_1-Y_1|^{n+sp}}\int_{\omega} 
		\Bigg(\int_{\omega}\frac{dY_2}
		{\left(1+\frac{|X_2-Y_2|^2}{|X_1-Y_1|^2}\right)^{\frac{n+sp}{2}}}\Bigg)
		|W(X_2)|^p\;dX_2dX_1dY_1\,.
		\end{align*}
		Using \cref{reduction formula} (ii) we get
		\begin{align*}
		\int_{\omega}\frac{dY_2}
		{\left(1+\frac{|X_2-Y_2|^2}{|X_1-Y_1|^2}\right)^{\frac{n+sp}{2}}}
		\leq &
		\int_{\R^{n-m}}\frac{dY_2}
		{\left(1+\frac{|X_2-Y_2|^2}{|X_1-Y_1|^2}\right)^{\frac{n+sp}{2}}}
		=|X_1-Y_1|^{n-m}\Theta_{n-m,n,p}.
		\end{align*}
		Applying this to the definition of $I_2$ and using the fact that $\|W\|_{L^p(\omega)}=1$, we obtain
		$$
		I_2^p \leq \frac{C_{n,s,p}\;\Theta_{n-m,n,p}}{2}\int_{\R^m}\int_{\R^m}\frac{\left|v_{\ell}(X_1)-v_{\ell}(Y_1)\right|^p}
		{|X_1-Y_1|^{m+sp}} dX_1dY_1=[v_{\ell}]_{s,p,\R^m}^p.
		$$
		By a change of variables in the definition of $v_\ell$, we get
	
		$$
		[v_{\ell}]_{s,p,\R^m}^p=\frac{\ell^{m-sp}}{\ell^m}[v]_{s,p,\R^m}^p=
		\frac{1}{\ell^{sp}}[v]_{s,p,\R^m}^p\quad\Rightarrow\quad I_2
		\leq \frac{[v]_{s,p,\R^m}}{\ell^{s}}.
		$$
		Now plugging the above finer estimates of $I_1$ and $I_2$ into \cref{main est}, we obtain
		$$[u_\ell]_{s,p,\Omega_{\infty}}\leq [W]_{s,p,\omega}+ \frac{[v]_{s,p,\R^m}}{\ell^{s}}.$$
		This finishes the proof of (1).  Proof of (2) is similar and hence omitted. 
	\end{proof}
	\smallskip

	In the next result we shall use the concept of weak formulation (see \cref{weak formulation}).
	
	\begin{lemma}\label{useful1}
		Let $x=(X_1,X_2)\in\Omega_{\infty}$ and define $u^*(x):=W(X_2)$, where $W$ is a weak solution of \cref{cross section EV}. Then 
		\begin{multline*}
		C_{n,s,p}\int_{\R^{n-m}}\int_{\rn}\frac{|u^*(x)-u^*(y)|^{p-2} (u^*(x)-u^*(y))\psi(X_2)}{|x-y|^{n+sp}}dydX_2\\
		=P^2_{n-m,s,p}(\omega)\int_{\omega}|W(X_2)|^{p-2}W(X_2)\psi(X_2)dX_2,\text{ for all }\psi\in W^{s,p}_{\omega}(\R^{n-m}).
		\end{multline*}	
	\end{lemma}
	\begin{proof}
		Let $\psi\in W^{s,p}_{\omega}(\R^{n-m})$. In the following calculation, we use the fact that $W$ is a weak solution of \cref{cross section EV}; also, we use (i) of \cref{reduction formula} in the second equality, and (ii) of the same lemma, with choices $a=|X_2-Y_2|$ and $z=X_1$, in the third equality. We, then, have:
		\begin{multline*}
		P^2_{n-m,s,p}(\omega)\int_{\omega}|W(X_2)|^{p-2}W(X_2)\psi(X_2)dX_2\\
		=C_{n-m,s,p}\int_{\R^{n-m}}\int_{\R^{n-m}}\frac{ |W(X_2)-W(Y_2)|^{p-2}(W(X_2)-W(Y_2))\psi(X_2)}{|X_2-Y_2|^{n-m+sp}}dY_2dX_2\\
		={C_{n,s,p}}\int_{\R^{n-m}}\int_{\R^{n-m}}\frac{ |W(X_2)-W(Y_2)|^{p-2}(W(X_2)-W(Y_2))\psi(X_2)}{|X_2-Y_2|^{n+sp}}\\
		\int_{\R^m}\frac{dY_1}{\left(1+\frac{|X_1-Y_1|^2}{|X_2-Y_2|^2}\right)^{\frac{n+sp}{2}}}dY_2dX_2\\
		=C_{n,s,p}\int_{\R^{n-m}}\int_{\R^m}\int_{\R^{n-m}}\frac{ |W(X_2)-W(Y_2)|^{p-2}(W(X_2)-W(Y_2))\psi(X_2)}{\left(|X_1-Y_1|^2+|X_2-Y_2|^2\right)^{\frac{n+sp}{2}}}dY_2dY_1dX_2\\
		=C_{n,s,p}\int_{\R^{n-m}}\int_{\rn}\frac{|u^*(x)-u^*(y)|^{p-2} (u^*(x)-u^*(y))\psi(X_2)}{|x-y|^{n+sp}}dydX_2.
		\end{multline*}
		Since $\psi$ is arbitrary, the lemma follows.
	\end{proof}
	\smallskip

	\begin{proof}[\textbf{\protect{Proof of \cref{best constants}}}]
		Suppose $W$ is the first eigenfunction corresponding to the first eigenvalue, $P^2_{n-m,s,p}(\omega)$ of the problem \cref{cross section EV}, which is strictly positive in $\omega$ (by \cref{simplicity}). Fix any $v\in C_c^\infty (\omegainfty)$ arbitrarily. Let $\{\rho_k\}$ be the standard mollifiers in $\R^{n-m}$. Now Define $\phi(X_1,X_2):=\frac{|v|^p}{W(X_2)^{p-1}}$, and $\phi_k(X_1,X_2):=\frac{|v|^p}{W_k(X_2)^{p-1}}$, where $W_k:=W*\rho_k$. At this point, we fix any $X_1 \in \R^{m}$ so that $v(X_1,\cdot)\in C_c^\infty(\omega)$, and $\phi_k(X_1,\cdot)\in C_c^1(\omega)\subset W^{s,p}_\omega(\R^{n-m})$ (by \cref{C1-subset}). Then, as $W>0$ in $\omega$, $W_k$ are strictly positive and smooth in $\omega$. Note that there exists $\alpha>0$ such that $W,W_k>\alpha$ in $\mbox{Supp }v(X_1,\cdot)$ for any $k$. Therefore, for any $X_2\in \mbox{Supp }v(X_1,\cdot)$, 
			\begin{multline*}
			|\phi_k(X_1,X_2)-\phi_k(X_1,Y_2)|=\left|\frac{|v(X_1,X_2)|^p}{W_k(X_2)^{p-1}}-\frac{|v(X_1,Y_2)|^p}{W_k(Y_2)^{p-1}}\right|\\
			=\left|\frac{ |v(X_1,X_2)|^p-|v(X_1,Y_2)|^p}{W_k^{p-1}(X_2)}+\frac{ |v(X_1,Y_2)|^p(W_k^{p-1}(Y_2)-W_k^{p-1}(X_2))}{W_k^{p-1}(X_2)W_k^{p-1}(Y_2)}\right|\\
			\leq \alpha^{p-1}\left||v(X_1,X_2)|^p-|v(X_1,Y_2)|^p\right|+||v||_\infty^p\frac{|W_k^{p-1}(Y_2)-W_k^{p-1}(X_2)|}{W_k^{p-1}(X_2)W_k^{p-1}(Y_2)}\\
			\leq p\alpha^{p-1}\left(|v(X_1,X_2)|^{p-1}+|v(X_1,Y_2)|^{p-1}\right)|v(X_1,X_2)-v(X_1,Y_2)|\\
			+(p-1)||v||_\infty^p\frac{(W_k^{p-2}(Y_2)+W_k^{p-2}(X_2))}{W_k^{p-1}(X_2)W_k^{p-1}(Y_2)}|W_k(X_2)-W_k(Y_2)|\\
			\leq C(p,\alpha,||v||_\infty)\left(|v(X_1,X_2)-v(X_1,Y_2)|+|W_k(X_2)-W_k(Y_2)|\right).
			\end{multline*}
			This shows that 
			\begin{multline*}
			\int_{\R^{n-m}} \int_{\R^{n-m}} \frac{|\phi_k(X_1,X_2)-\phi_k(X_1,Y_2)|^p}{|X_2-Y_2|^{n-m+sp}}dX_2dY_2 \\
			\leq \int_{\R^{n-m}} \int_{\R^{n-m}} C(p,W,v) \frac{|W_k(X_2)-W_k(Y_2)|^p}{|X_2-Y_2|^{n-m+sp}}dX_2dY_2 + C(p,W,v) [v(X_1,\cdot)]^p_{s,p,\R^{n-m}}.
			\end{multline*}
			Now one can easily check that $W_k$ converges to $W$ in $W^{s,p}_\omega(\R^{n-m})$ (see \cite[Lemma~11]{FiSeVa}) and also pointwise. We can apply generalised dominated convergence theorem (see \cite[Theorem~19, Section~4.4]{RoFi}) to conclude that $\phi\in W^{s,p}_\omega(\R^{n-m})$. 
			Define $u(X_1,X_2):=W(X_2)$ and apply discrete Picone inequality \cref{picone} on $u$ and $|v|$ to obtain
			\begin{multline*}
			|u(X_1,X_2)-u(Y_1,Y_2)|^{p-2}(u(X_1,X_2)-u(Y_1,Y_2))(\phi(X_1,X_2)-\phi(Y_1,Y_2))\\
			\leq |v(X_1,X_2)-v(Y_1,Y_2)|^p.
			\end{multline*}
			This gives
			\begin{multline}\label{eq1}
			\frac{C_{n,s,p}}{2}\int_{\R^n}\int_{\R^n}\frac{|W(X_2)-W(Y_2)|^{p-2}(W(X_2)-W(Y_2))(\phi(x)-\phi(y))}{|x-y|^{n+sp}}dxdy\\
			\leq \frac{C_{n,s,p}}{2} \int_{\R^n}\int_{\R^n}\frac{|v(x)-v(y)|^p}{|x-y|^{n+sp}}dxdy.
			\end{multline}
			Now, observe the following calculation, where we have used Fubini's theorem, \cref{useful1} and positiveness of $W$:
			\begin{multline*}
			C_{n,s,p} \int_{\R^n}\int_{\R^n}\frac{|W(X_2)-W(Y_2)|^{p-2}(W(X_2)-W(Y_2))\phi(x)}{|x-y|^{n+sp}}dxdy\\
			= \int_{\R^{m}} C_{n,s,p}\int_{\R^{n-m}}\int_{\R^n}\frac{|W(X_2)-W(Y_2)|^{p-2}(W(X_2)-W(Y_2))\phi(X_1,X_2)}{|x-y|^{n+sp}}dydX_2dX_1\\
			= P^2_{n-m,s,p}(\omega)\int_{\R^{m}} \int_{\omega}|W(X_2)|^{p-2}W(X_2)\phi(X_1,X_2)dX_2dX_1\\
			= P^2_{n-m,s,p}(\omega)\int_{\Omega_{\infty}} |v(x)|^pdx <\infty.
			\end{multline*}
			This finiteness of the integrand allows us to rewrite it as: 
			\begin{multline*}
			\frac{C_{n,s,p}}{2} \int_{\R^n}\int_{\R^n}\frac{|W(X_2)-W(Y_2)|^{p-2}(W(X_2)-W(Y_2))(\phi(x)-\phi(y)}{|x-y|^{n+sp}}dxdy\\
			= C_{n,s,p} \int_{\R^n}\int_{\R^n}\frac{|W(X_2)-W(Y_2)|^{p-2}(W(X_2)-W(Y_2))\phi(x)}{|x-y|^{n+sp}}dxdy\\
			= P^2_{n-m,s,p}(\omega)\int_{\Omega_{\infty}} |v(x)|^pdx,
			\end{multline*}
			which, when combined with \cref{eq1}, gives
			$$
			P^2_{n-m,s,p}(\omega)\int_{\Omega_{\infty}} |v(x)|^pdx \leq \frac{C_{n,s,p}}{2} \int_{\R^n}\int_{\R^n}\frac{|v(x)-v(y)|^p}{|x-y|^{n+sp}}dxdy.
			$$
		As this is true for any $v\in C_c^\infty(\Omega_{\infty})$, we have $P^2_{n-m,s,p}(\omega)\leq P^2_{n,s,p}(\Omega_\infty)$, by density of $C_c^\infty(\Omega_{\infty})$ in $W^{s,p}_{\Omega_\infty}(\rn)$. The upper bound of $P^2_{n,s,p}(\omegainfty)$ follows from (2) of \cref{useful prop}. 
		\smallskip
		
		Now for the last part of the theorem, suppose that there exist a function $u$ such that $P^2_{n,s,p}(\omegainfty)=\frac{[u]_{s,p.\rn}^p}{\int_{\omegainfty}|u(x)|^pdx}$. Then $u$ is a weak solution of the problem
		\begin{equation}\label{Ev on omegainfinity}
		\begin{cases}
		(-\Delta_{n,p})^su=P_{n,s,p}^2(\Omega_{\infty})|u|^{p-2}u \text{ in }\Omega_{\infty},
		\\
		u=0 \text{ in }\R^{n}\setminus\Omega_{\infty}.
		\end{cases}
		\end{equation}
		In other words, $u$ is an eigenfunction corresponding to the first eigenvalue $P^2_{n,s,p}(\omegainfty)$. Let $h \in\R^{m}$, define $v_{h}(x)=u(X_1+h,X_2)$. By change of variable, we also have $v_h$ is an eigenfunction of \cref{Ev on omegainfinity} associated to the eigenvalue $P^2_{n,s,p}(\omegainfty)$ for any $h$.  Since $P^2_{n,s,p}(\omegainfty)$ is simple (see \cref{simplicity}), $u=\alpha_h v_h$ for some constant $\alpha_h.$ Therefore, by a change of variable, we have 
		\begin{multline*}
		\int_{\omegainfty}|u|^pdx=\int_{\R^m}\int_{ \omega}|u(X_1,X_2)|^pdX_2dX_1=|\alpha_h|^p\int_{\R^m}\int_{ \omega}|v_h(X_1,X_2)|^pdX_2dX_1
		=|\alpha_h|^p\int_{\omegainfty}|u|^pdx
		\end{multline*}
		Thus, we get $|\alpha_h|^p=1$ and this imply that $\alpha_h=1$, because $u$ has constant sign in $\omegainfty.$ Therfore, we get $u(X_1,X_2)=v_h(X_1,X_2)$ for any $h\in\R^m$. Hence $u$ is independent of $X_1$ variable. In particular,  $||u||_{L^p(\omegainfty)}$ is infinite, which gives a contradiction. This completes the proof of \cref{best constants}. 
	\end{proof}
	\smallskip
	
\begin{corollary}
		Let $\{\Omega_\ell\}$ be an increasing sequence of bounded open sets in $\rn$ that is $\Omega_{\ell}\subseteq\Omega_{\ell_1}$ for any $0<\ell<\ell_1$. If $\Omega=\bigcup_{\ell>0}\Omega_\ell$. Then we have 
		$$
		P^2_{n,s,p}(\Omega)=\inf\limits_{\ell>0}P^2_{n,s,p}(\Omega_\ell).
		$$
	\end{corollary}
	\begin{proof}
		By domain monotonicity property (2) of \cref{prop:elementary properties}, we have $\inf\limits_{\ell>0}P^2_{n,s,p}(\Omega_\ell)\geq P^2_{n,s,p}(\Omega)$. So, to establish the result, we only need to show $\inf\limits_{\ell>0}P^2_{n,s,p}(\Omega_\ell)\leq P^2_{n,s,p}(\Omega)$. Now, for any $v\in C_c^\infty(\Omega)$, there exists an $\ell>0$, big enough, such that $\mbox{supp}(v)\subset\Omega_\ell$. Then we have $\|v|_{\Omega_\ell}\|_{L^p(\Omega_\ell)}=\|v\|_{L^p(\Omega)}$ and $[v|_{\Omega_\ell}]_{s,p,\rn}=[v]_{s,p,\rn}$. So $\inf\limits_{\ell>0}P^2_{n,s,p}(\Omega_\ell)\leq P^2_{n,s,p}(\Omega_\ell)\leq\frac{[v]_{s,p,\rn}^p}{\|v\|_{L^p(\Omega)}^p}$. Since this holds for any $v\in C_c^\infty(\Omega)$, we conclude $\inf\limits_{\ell>0}P^2_{n,s,p}(\Omega_\ell)\leq P^2_{n,s,p}(\Omega)$.
	\end{proof}
	\begin{lemma}
		\label{lemma:angle condition}
		Let $\frac{1}{p}<s<1$, $\Omega\subset\R^n$ be a measurable set, and $f:\mathbb{S}^{n-1}\to [0,\infty)$ be an $\mathcal{H}^{n-1}$-measurable function satisfying 
		$$
		P^1_{1,s,p}\left(\{ t\in\R :\ x+tw \in \Omega\}\right)\geq f(w)
		$$
		for a.e. $w\in\mathbb{S}^{n-1}$ and a.e. $x\in \{y\in \R^n:\,y\cdot w=0\}.$ Then
		$$
		P^1_{n,s,p}(\Omega)\geq \frac{C_{n,s,p}}{2C_{1,s,p}}\int_{\mathbb{S}^{n-1}}f(w)d\mathcal{H}^{n-1}(w).
		$$
	\end{lemma}
	\begin{proof}
		Let us choose $w\in\mathbb{S}^{n-1}$ and $x\in L_w:= \{y\in \R^n:\,y\cdot w=0\}$ arbitrarily. Denote
		$
		\Omega_{w,x}:=\{t\in\R  : \ x+t w\ \in \Omega\}.
		$
		Then from the hypotheses, we have
		\begin{align*}
		\frac{C_{1,s,p}}{2}\int_{\{\ell:\,x+\ell w\in\Omega\}}
		\int_{\{t:\,x+tw\in\Omega\}}\frac{|u(x+\ell w)-u(x+tw)|^p}{|\ell-t|^{1+sp}}dtd\ell
		\geq 
		&P^1_{1,s,p}(\Omega_{w,x})\int_{\Omega_{w,x}}|u(x+tw)|^pdt
		\smallskip
		\\
		\geq & f(w) \int_{\Omega_{w,x}}|u(x+tw)|^pdt.
		\end{align*}
		We apply Fubini's theorem to get, for any $w\in \mathbb{S}^{n-1}$,
		$$
		\int_{L_w}d\mathcal{H}^{n-1}(x)\int_{\Omega_{x,w}}|u(x+tw)|^pdt=
		\int_{\Omega}|u|^p,
		$$
		which, along with \cref {lemma:Loss Sloan}, gives
		$$
		[u]_{s,p,\Omega}^p\geq \frac{C_{n,s,p}}{2C_{1,s,p}}
		\left( \int_{\mathbb{S}^{n-1}}f(w)d\mathcal{H}^{n-1}(w)\right)\int_{\Omega}|u|^p.
		$$
		This proves the lemma.
	\end{proof}
	\smallskip
	
	Before proving \cref{Poincare 1}, observe that for any function $f$ depending only on $\sigma_1$, where $\sigma=(\sigma_1,\cdots,\sigma_{n-1})$, we have, from the discussion on \nameref{hyper} that
	\begin{align*}
	\int_{A_{n-1}}f(\sigma_1)g_{n-1}(\sigma)d\sigma=
	&\int_0^{\pi}
	f(\sigma_1)(\sin\sigma_1)^{n-2}\left(\int_{Q_{n-2}} 
	(\sin\sigma_2)^{n-3}\cdots\sin\sigma_{n-2}d\sigma_2\cdots d\sigma_{n-2}\right)
	d\sigma_1
	\smallskip
	\\
	=&\int_0^{\pi} f(\sigma_1)(\sin\sigma_1)^{n-2}\left(\int_{Q_{n-2}} 
	g_{n-2}(\theta)d\theta\right)
	d\sigma_1
	\smallskip
	\\
	=&\mathcal{H}^{n-2}(\mathbb{S}^{n-2})
	\int_0^{\pi} f(\sigma_1)(\sin\sigma_1)^{n-2}d\sigma_1.
	\end{align*}
	In particular, using \cref{eq:surface of S n-1,eq:properties of Beta}, for $f(\sigma)=|\cos\sigma_1|^{sp}$ we obtain
	\begin{equation}
	\label{eq:cos varphi 1 integral}
	\begin{split}
	\int_{A_{n-1}}|\cos\sigma_1|^{sp}g_{n-1}(\sigma)d\sigma=&
	2\mathcal{H}^{n-2}(\mathbb{S}^{n-2})
	\int_0^{\frac{\pi}{2}} (\cos\sigma_1)^{sp}(\sin\sigma_1)^{n-2}d\sigma_1
	\smallskip
	\\
	=&\frac{2\pi^{\frac{n-1}{2}}}{\Gamma\left(\frac{n-1}{2}\right)}
	B\left(\frac{n-1}{2},\frac{sp+1}{2}\right).  
	\end{split}  
	\end{equation}
	
	Now we are ready to prove \cref{Poincare 1}.
	\smallskip
	
	\begin{proof}[\textbf{Proof of \cref{Poincare 1}}]
		\textit{Part (1):}
		Assume $s\in (0, \frac{1}{p}]$. Note that the case $p=1$ is covered here, with $sp<1$. We apply \cref{useful prop} with $m=n-1$ and $\omega=(-1,1) \subset \re$ to deduce 
		$P^1_{n,s,p}(\omegainfty)\leq P^1_{1,s,p}((-1,1)).$
		Now applying \cref{two spaces equal} on $P^1_{1,s,p}((-1,1))=0$ we get the result.
		\smallskip
		
		\textit{Part (2):} Let us assume $s\in(\frac{1}{p},1)$. We know, from \cref{useful prop}, that $P_{n,s,p}^1(\Omega_{\infty})\leq P_{1,s,p}^1((-1,1))$. So, it is enough to prove that
		\begin{equation}
		\label{eq:P omegainfty geq P crosssec}
		P_{n,s,p}^1(\Omega_{\infty})\geq P_{1,s,p}^1((-1,1)).
		\end{equation}
		We shall show this using \cref{lemma:angle condition}. Choose $w=(w_1,\cdots,w_n)\in \mathbb{S}^{n-1}$ and $x\in\R^{n}$ such that $w_1\neq 0$ and $x.w=0$. Notice that $\mathcal{L}^1(\{t\in\R: \, x+tw\in\Omega_\infty\})$, i.e. the length of the intersection $\Omega_{\infty}\cap \{x+tw:\,t\in\R\}$, is independent of $x\in \omega^\perp$. So we have
		\begin{multline*}
		\mathcal{L}^1(\{t\in\R: \, x+tw\in\Omega_\infty \})
		=\mathcal{H}^1\left(\Omega_{\infty}\cap \{x+tw:\,t\in\R\}\right)\\
		=
		\mathcal{H}^1\left(\Omega_{\infty}\cap \{(-1,0,\ldots,0)+tw:\,t\in\R\}\right)
		=| t_0(w)|,
		\end{multline*}
		where $-1+t_0(w)w_1=1$ i.e. $t_0(w)=\frac{2}{w_1}$.	From (ii) of \cref{prop:elementary properties} we see that
		$$
		P_{1,s,p}^1\left(\{t\in\R: \, x+tw\in\Omega_\infty \}\right)
		=\left(\frac{|w_1|}{2}\right)^{sp}P_{1,s,p}^1((0,1))
		=|w_1|^{sp}P_{1,s,p}^1((-1,1)).
		$$
			The above equality enables us to apply \cref{lemma:angle condition}, with the choice $f(w)=P_{1,s,p}^1((-1,1)) |w_1|^{sp}$. We get
			$$
			P_{n,s,p}^1(\omegainfty)
			\geq P_{1,s,p}^1((-1,1))\frac{C_{n,s,p}}{2 C_{1,s,p}}\int_{\mathbb{S}^{n-1}}
			|w_1|^{sp}d\mathcal{H}^{n-1}.
			$$
			Again, using \nameref{hyper}, in the RHS of the above inequality, we get
			$$
			P_{n,s,p}^1(\omegainfty)
			\geq
			P_{1,s,p}^1((-1,1))\frac{C_{n,s,p}}{2 C_{1,s,p}}\int_{A_{n-1}}
			|\cos\sigma_1|^{sp}g_{n-1}(\sigma)d\sigma.
			$$
			Now, \cref{eq:cos varphi 1 integral} gives
			$$
			P_{n,s,p}^1(\omegainfty)
			\geq
			P_{1,s,p}^1((-1,1))\frac{C_{n,s,p}}{2 C_{1,s,p}}
			\frac{2\pi^{\frac{n-1}{2}}}{\Gamma\left(\frac{n-1}{2}\right)}
			B\left(\frac{n-1}{2},\frac{sp+1}{2}\right).
			$$
		
		Again, using \cref{const flap,eq:properties of Beta}, we find that
		$$
		\frac{C_{n,s,p}}{ C_{1,s,p}}
		\frac{\pi^{\frac{n-1}{2}}}{\Gamma\left(\frac{n-1}{2}\right)}
		B\left(\frac{n-1}{2},\frac{sp+1}{2}\right)=1,
		$$
		consequently $P_{n,s,p}^1(\omegainfty) \geq P_{1,s,p}^1((-1,1))$. This concludes the proof of  \cref{eq:P omegainfty geq P crosssec} and hence the theorem follows.  
	\end{proof}
	
	\begin{proof}[\textbf{Proof of \cref{l to infinity}}]
		The domain monotonicity property ((i) of \cref{prop:elementary properties}) and \cref{best constants}, implies $P^2_{n-m,s,p}(\omega)\leq P^2_{n,s,p}(\Omega_{\ell})$. For the reverse inequality, following the same proof as in (1) of \cref{useful prop}, where the domain of integration $\Omega_{\infty}$ is replaced by $\Omega_{\ell}$, we obtain
		\begin{multline*}
			P^2_{n,s,p}(\Omega_\ell)
			\leq \left(P^2_{n-m,s,p}(\omega)^\frac{1}{p}+ \frac{[v]_{s,p,\R^m}}{\ell^{s}}\right)^p\\
			\leq P^2_{n-m,s,p}(\omega)+p2^{p-1}\left( \frac{P^2_{n-m,s,p}(\omega)^\frac{p-1}{p}[v]_{s,p,\R^m}}{\ell^{s}} +\frac{[v]_{s,p,\R^m}^p}{\ell^{sp}} \right)
			= P^2_{n-m,s,p}(\omega)+\frac{C_1}{\ell^s}+\frac{C_2}{\ell^{sp}},
		\end{multline*}
	where we used the following elementary inequality: $(a+b)^q\leq a^q+q2^{q-1}(a^{q-1}b+b^q)$ for $a,b\geq 0$ and $q\geq1$. Combining these two estimates of $P^2_{n,s,p}(\Omega_{\ell})$, the first part of the theorem follows. Now letting $\ell\to\infty$ and applying \cref{best constants} we conclude the last equality. This finishes the proof of \cref{l to infinity}. 
		
	\end{proof}
	
	\smallskip
	
	\textbf{Acknowledgment:} The authors would like to thanks Prof. Prosenjit Roy, Prof. Gyula Csat\'o and Dr. Indranil Chowdhury for fruitful discussions on this subject.

	
\end{document}